\documentclass[12pt]{article}
\usepackage{a4wide}
\usepackage{amsthm}
\usepackage{amsfonts}
\usepackage{amssymb}
\usepackage{amsmath}
\usepackage{cite}
\usepackage{epsfig}
\usepackage{tabularx}

\newtheorem{theorem}{Theorem}
\newtheorem{corollary}[theorem]{Corollary}
\newtheorem{lemma}[theorem]{Lemma}
\newtheorem{proposition}[theorem]{Proposition}

\newcommand\TT{{\cal T}}
\newcommand\PP{{\cal P}}

\DeclareTextCompositeCommand{\v}{OT1}{l}{l\nobreak\hspace{-.1em}'}
\DeclareTextCompositeCommand{\v}{OT1}{t}{t\nobreak\hspace{-.1em}'\nobreak\hspace{-.15em}}

\begin{document}
\title{Twin-width of graphs on surfaces\thanks{The first two authors have been supported by the MUNI Award in Science and Humanities (MUNI/I/1677/2018) of the Grant Agency of Masaryk University. The third author is supported by the Slovenian Research Agency (research program P1-0383, research projects J1-3002, J1-4008).}}

\author{Daniel Kr{\'a}\v{l}\thanks{Faculty of Informatics, Masaryk University, Botanick\'a 68A, 602 00 Brno, Czech Republic. E-mails: {\tt dkral@fi.muni.cz} and {\tt kristyna.pekarkova@mail.muni.cz}.}\and
\newcounter{lth}
\setcounter{lth}{2}
        Krist\'yna Pek\'arkov\'a$^\fnsymbol{lth}$\and
	Kenny \v Storgel\thanks{Faculty of Information Studies in Novo mesto, Ljubljanska cesta 31a, 8000 Novo mesto, Slovenia. E-mail: {\tt kennystorgel.research@gmail.com}.}}
\date{} 
\maketitle

\begin{abstract}
Twin-width is a width parameter introduced by Bonnet, Kim, Thomass\'e and Watrigant [FOCS'20, JACM'22],
which has many structural and algorithmic applications.
We prove that the twin-width of every graph embeddable in a surface of Euler genus $g$
is at most $18\sqrt{47g}+O(1)$, which is asymptotically best possible as
it asymptotically differs from the lower bound by a constant multiplicative factor.
Our proof also yields a quadratic time algorithm to find a corresponding contraction sequence.
To prove the upper bound on twin-width of graphs embeddable in surfaces,
we provide a stronger version of the Product Structure Theorem for graphs of Euler genus $g$ that
asserts that every such graph is a subgraph of the strong product of a path and a graph with a tree-decomposition
with all bags of size at most eight with a single exceptional bag of size $\max\{8,32g-27\}$.
\end{abstract}

\section{Introduction}
\label{sec:intro}

Twin-width is a graph parameter,
which has recently been introduced by Bonnet, Kim, Thomass\'e and Watrigant~\cite{BonKTW20,BonKTW22}.
It has quickly become one of the most intensively studied graph width parameters
due to its many connections to algorithmic and structural questions in both computer science and mathematics.
In particular,
classes of graphs with bounded twin-width (we refer to Section~\ref{sec:prelim} for the definition of the parameter)
include at the same time well-structured classes of sparse graphs and well-structured classes of dense graphs.
Particular examples are classes of graphs with bounded tree-width,
with bounded rank-width (or equivalently with bounded clique-width), and
classes excluding a fixed graph as a minor.
As the first order model checking is fixed parameter tractable for classes of graphs with bounded twin-width~\cite{BonKTW20,BonKTW22},
the notion led to a unified view of various earlier results on fixed parameter tractability of first order model checking of graph properties~\cite{Cou90,CouMR00,DvoKT10,DvoKT13,FriG99,FriG01,See96}, and
more generally first order model checking properties of other combinatorial structures such as
matrices, permutations and posets~\cite{BalH21,BonGOSTT22,BonNOST21,BonNOST23eurocomb}.

The foundation of the theory concerning twin-width has been laid by Bonnet, Kim, Thomass\'e and their collaborators
in a series of papers~\cite{BonGKTW21b,BonGKTW21c,BonGOSTT22,BonGOT22stacs,BonKRT22,BonGTT22,BonCKKLT22,BonKTW20,BonKTW22},
also see~\cite{Tho22}.
The amount of literature on twin-width is rapidly growing and
includes exploring algorithmic aspects of twin-width~\cite{BerBD22,BonGKTW21c,BonGOSTT22,BonKRTW21,PilSZ22},
combinatorial properties~\cite{AhnHKO22,BalH21,BerGGHPS23,BonGKTW21c,BonKRT22,DreGJOR22,PilS23}, and
connections to logic and model theory~\cite{BonGOSTT22,BonKTW20,BonKTW22,BonNOST21,BonNOST23eurocomb,GajPT22}.
While it is known that many important graph classes have bounded twin-width,
good bounds are known only in a small number of specific cases.
One of the examples is the class of graphs of bounded tree-width
where an asymptotically optimal bound, exponential in tree-width, was proven by Jacob and Pilipczuk~\cite{JacP22wg}.
Another example is the class of planar graphs.
The first explicit bound of $583$ by Bonnet, Kwon and Wood~\cite{BonKW22}
was gradually improved in a series of papers~\cite{JacP22wg,BekLHK22,Hli22}
culminating with a bound of $8$ obtained by Hlin\v en\'y and Jedelsk\'y~\cite{HliJ23};
also see~\cite{Hli23,Hli23eurocomb} for a simpler proof of this bound and
\cite{Jed24} for a promising approach of obtaining the upper bound of $7$,
which would be tight
since Lamaison and the first author~\cite{KraL23online} presented a construction of a planar graph with twin-width $7$.
In this paper, 
we extend this list by providing an asymptotically optimal upper bound on the twin-width of graphs embeddable in surfaces of higher genera.
Specifically, we prove the following two results (the latter is used to prove the former):
\begin{itemize}
\item We show that the twin-width of a graph embeddable in a surface of Euler genus $g$ is at most $18\sqrt{47g}+O(1)$,
      which is asymptotically best possible;
      our proof also yields a quadratic time algorithm to find a witnessing sequence of vertex contractions.
\item We provide a strengthening of the Product Structure Theorem for graphs embeddable in a surface of Euler genus $g$
      by showing that such graphs are subgraphs of a strong product of a path and a graph that almost has a bounded tree-width.
\end{itemize}
We next present the two results in more detail while also presenting the related existing results, and
discuss algorithmic aspects of our results.

\subsection{Twin-width of graphs embeddable in surfaces}

Graphs that can be embedded in surfaces of higher genera, such as the projective plane, the torus and the Klein bottle,
form important minor-closed classes of graphs with many applications and connections~\cite{MohT01}.
While the general theory concerning minor-closed classes of graphs
yields that graphs embeddable in a fixed surface have bounded twin-width,
the bounds are quite enormous:
the results from~\cite[Section 4]{BonKRT22} on $d$-contractible graphs (graphs embeddable in a surface of Euler genus $g$
are $O(g)$-contractible~\cite{Iva92}) yields a bound double exponential in $g$, and
the Product Structure Theorem for graphs embeddable in surfaces~\cite{DujJMMUW19,DujJMMUW20}
together with results on the twin-width of graphs with bounded tree-width~\cite{JacP22wg},
of the strong product of graphs~\cite{PetS23jour} and their subgraph closure~\cite{BonGKTW21b}
yields an exponential bound.

Bonnet, Kwon and Wood~\cite{BonKW22} showed that
every graph embeddable in a surface of Euler genus $g$ has twin-width at most $205g+583$.
Our main result asserts that twin-width of every graph that can be embedded in a surface of Euler genus $g$
is at most $18\sqrt{47g}+O(1)\approx 123.4\sqrt{g}+O(1)$.
This bound is asymptotically optimal as any graph with $\sqrt{6g}-O(1)$ vertices can be embedded in a surface of Euler genus $g$ and
the $n$-vertex Erd\H os-R\'enyi random graph $G_{n,1/2}$ has twin-width at least $n/2-O(\sqrt{n\log n})$~\cite{AhnCHKO22},
i.e., there exists a graph with twin-width $\sqrt{3g/2}-o(\sqrt{g})$ embeddable in a surface of Euler genus $g$ (a short
proof is given for completeness in Section~\ref{sec:lower}).
In particular, our upper bound asymptotically differs from the lower bound by a multiplicative factor $6\sqrt{282}\approx 100.76$.

We remark that
we are aware of several parts of our argument
whose refinement would lead to a decrease of the multiplicative constant in the upper bound (with the resulting
multiplicative constant to be around $20$).
However, we have decided not to do so due to the technical nature of such refinements and
the absence of additional structural insights gained by the refinements.

\subsection{Product Structure Theorem}

To prove our main result,
we prove a modification of the Product Structure Theorem that applies to graphs embeddable in surfaces.
The Product Structure Theorem is a recent significant structural result introduced by
Dujmovi\'c, Joret, Micek, Morin, Ueckerdt and Wood~\cite{DujJMMUW19,DujJMMUW20},
which brought new substantial insights into the structure of planar graphs and
led to breakthroughs on several long standing open problems concerning planar graphs, see, e.g.~\cite{DujEJWW20}.
We also refer to the survey by Dvo\v r\'ak et al.~\cite{DvoHJLW21} on the topic.
The statement of the Product Structure Theorem originally proven by Dujmovi\'c et al.~\cite{DujJMMUW19,DujJMMUW20}
reads as follows (we remark that the statement in~\cite{DujJMMUW19,DujJMMUW20} does not include the condition
on planarity of the graph of bounded tree-width, however, an easy inspection of the proof yields this).

\begin{theorem}
\label{thm:prod8}
Every planar graph is a subgraph of the strong product of a path and a planar graph with tree-width at most $8$.
\end{theorem}

Ueckerdt et al.~\cite{UecWY22} improved the result as follows (in fact, we state a corollary of their main result
to avoid defining the notion of simple tree-width that is not needed in our further presentation).

\begin{theorem}
\label{thm:prod6}
Every planar graph is a subgraph of the strong product of a path and a planar graph with tree-width at most $6$.
\end{theorem}

Dujmovi\'c et al.~\cite{DujJMMUW19,DujJMMUW20} also proved two extensions of the Product Structure Theorem
to graphs embeddable in surfaces.

\begin{theorem}
\label{thm:prodgen}
Every graph embeddable in a surface of Euler genus $g>0$
is a subgraph of the strong product of a path, the complete graph $K_{2g}$ and
a planar graph with tree-width at most $9$.
\end{theorem}

\begin{theorem}
\label{thm:prodgen2}
Every graph embeddable in a surface of Euler genus $g>0$
is a subgraph of the strong product of a path, the complete graph $K_{\max\{2g,3\}}$ and
a planar graph with tree-width at most $4$.
\end{theorem}

A stronger version was proven by Distel at el.~\cite{DisHHW22};
the discussion of the even stronger statement implied by the proof of the next theorem given in~\cite{DisHHW22}
can be found after Theorem~\ref{thm:prodnew}.

\begin{theorem}
\label{thm:prodgenP}
Every graph embeddable in a surface of Euler genus $g>0$
is a subgraph of the strong product of a path, the complete graph $K_{\max\{2g,3\}}$ and
a planar graph with tree-width at most $3$.
\end{theorem}

We remark that it is not possible to replace $K_{2g}$ in the statement of Theorems~\ref{thm:prodgen}, \ref{thm:prodgen2} and~\ref{thm:prodgenP}
with a complete graph of sublinear order as long as the bound on the tree-width stays constant
since the layered tree-width of graphs embeddable in a surface of Euler genus $g$ is linear in $g$~\cite{DujMW17} (the definition
of layered tree-width is given in Section~\ref{sec:prelim}).
To prove our upper bound on the twin-width of graphs embeddable in surfaces,
we strengthen the statement of the Product Structure Theorem for graphs embeddable in surfaces as follows.
Theorems~\ref{thm:prodgen}, \ref{thm:prodgen2} and~\ref{thm:prodgenP} imply that
every graph embeddable in a surface of Euler genus $g>0$
is a subgraph of the strong product of a path and a graph with tree-width at most $20g-1$, $\max\{10g-1,14\}$ and $\max\{8g-1,11\}$, respectively.
The next theorem, which we prove in Section~\ref{sec:product}, asserts that
it is possible to assume that the tree-width of the graph in the product is almost at most $7$
in the sense that all bags except possibly for a single bag have size at most $8$.

\begin{theorem}
\label{thm:prodnew}
Every graph embeddable in a surface of Euler genus $g>0$
is a subgraph of the strong product of a path and a graph $H$ that
has a rooted tree-decomposition such that
\begin{itemize}
\item the root bag has size at most $\max\{8,32g-27\}$, and
\item every bag except the root bag has size at most $8$.
\end{itemize}
\end{theorem}

We remark that,
similarly as in Theorem~\ref{thm:prodgen} it is not possible to replace $K_{2g}$ with a complete graph of smaller order,
it is necessary to permit at least one of the bags to have a size linear in $g$ in Theorem~\ref{thm:prodnew}
since the layered tree-width of graphs embeddable in a surface of Euler genus $g$ is linear in $g$~\cite{DujMW17}.
Hence, the statement of Theorem~\ref{thm:prodnew} is the best possible asymptotically.

It is interesting to note the proof of Theorem~\ref{thm:prodgenP} given in~\cite{DisHHW22}
implies that every graph embeddable in a surface of Euler genus $g>0$
is a subgraph of the strong product of a path and a graph that
can be obtained from a planar graph with tree-width at most $3$
by replacing one vertex of this planar graph with $K_{2g}$ and
the remaining vertices with $K_3$ (and replacing each edge of the planar graph
with a complete bipartite graph between the corresponding sets of vertices).
However, the vertex of the planar graph that is replaced with $K_{2g}$
can be contained in many bags of the tree-decomposition and
so the proof given in~\cite{DisHHW22} does not yield a statement similar to that of Theorem~\ref{thm:prodnew}
since the number of bags in the tree-decomposition with size linear in $g$ can be arbitrary (although
each such bag contains the same $2g$ vertices of $K_{2g}$ in addition to $9$ other vertices).
%%Likewise, the proofs of Theorems~\ref{thm:prodgen} and~\ref{thm:prodgen2} given in~\cite{DujJMMUW19,DujJMMUW20}
%do not provide control over the sizes of bags sufficient to imply the statement of Theorem~\ref{thm:prodnew}.
The main new component in the proof of Theorem~\ref{thm:prodnew} (compared to the proofs given in~\cite{DisHHW22,DujJMMUW19,DujJMMUW20})
is Lemma~\ref{lm:twout} given in Section~\ref{sec:product},
which is crucial so that we are able to restrict the sizes of all but one bag in a tree-decomposition to a constant size.

We also note the following corollary of Theorem~\ref{thm:prodnew} for projective planar graphs.

\begin{corollary}
\label{cor:projective}
Every graph embeddable in the projective plane
is a subgraph of the strong product of a path and a graph with tree-width at most $7$.
\end{corollary}

\subsection{Algorithmic aspects}

We decided to present our results in purely structural way,
i.e., focus on establishing the bounds without giving an algorithm in parallel.
However, since all the proofs that we present are algorithmic,
we also obtain a quadratic time algorithm (when the genus $g>0$ is fixed) that
given a graph $G$ embeddable in a surface of genus $g$,
constructs a sequence of contractions witnessing that the twin-width of $G$ is at most $18\sqrt{47g}+O(1)$,
i.e., the red degree of trigraphs obtained during contractions does not exceed the bound given in Theorem~\ref{thm:main}.
We remark that we measure the time complexity of the algorithm in terms of the number of vertices in $G$;
note that the number of edges of an $n$-vertex graph embeddable in a surface of Euler genus $g$ is at most $3n+3g-6$.

We now discuss the algorithm in more detail following the steps in the proof of Theorem~\ref{thm:main}.
Since it is possible to find an embedding of a graph in a fixed surface in linear time~\cite{Moh96,Moh99},
we can assume that the input graph $G$ is given together with its embedding in the surface (recall that
for every $g\ge 2$, there are only two non-homeomorphic surfaces of Euler genus $g$).
When the embedding of $G$ in the surface is fixed,
we complete it to a triangulation $G'$ (we permit adding parallel edges if needed).
We next choose an arbitrary BFS spanning tree $T$ of $G'$ and
identify $g$ edges $a_1b_1,\ldots,a_gb_g$ as described in Lemma~\ref{lm:BFS}, which was proven in~\cite{DujJMMUW19,DujJMMUW20}.
The proof of Lemma~\ref{lm:BFS} in~\cite{DujJMMUW19,DujJMMUW20}
proceeds by constructing a spanning tree in the dual graph that avoids the edges of $T$ and
choosing the edges contained in neither $T$ nor the spanning tree of the dual graph as
the edges $a_1b_1,\ldots,a_gb_g$;
this can be implemented in linear time.
When the edges $a_1b_1,\ldots,a_gb_g$ are fixed,
the construction of the walk $W$ and the vertical paths described in Lemma~\ref{lm:BFS-paths} requires linear time.

We next need to identify the vertical paths described in Lemma~\ref{lm:twout} that
split the near-triangulation bounded by $W$ obtained from $G'$ into parts,
each bounded by at most six vertical paths.
This may require processing the near-triangulation repeatedly following the steps in the inductive proof of Lemma~\ref{lm:twout},
however, each step can be implemented in linear time and the number of steps is also at most linear.
Finally, we apply the recursive procedure described in Lemma~\ref{lm:tw7}
to each of the parts delimited by faces of the $2$-connected graph obtained in Lemma~\ref{lm:twout};
again, the number of steps in the recursive procedure is linear and each can be implemented in linear time, and
they directly yield the collection $\PP$ of vertical paths and
the tree-decomposition $\TT$ of $G'/\PP$ described in Theorem~\ref{thm:product}.
Since the paths $\PP$ and the tree-decomposition $\TT$ fully determine the order of the contraction of the vertices and
the order can be easily determined in linear time following the proof of Theorem~\ref{thm:main},
we conclude that there is a quadratic time algorithm that constructs a sequence of contractions such that
the red degree of trigraphs obtained during contractions does not exceed the bound given in Theorem~\ref{thm:main}.

We would like to remark that we have not attempted to optimize the running time of the algorithm,
which would particularly require to refine the recursive steps in the proofs of Lemmas~\ref{lm:tw7} and~\ref{lm:twout}.

\section{Preliminaries}
\label{sec:prelim}

In this section, we introduce notation used throughout the paper.
We use $[n]$ to denote the set of the first $n$ positive integers, i.e., $\{1,\ldots,n\}$.
All graphs considered in this paper are simple and have no parallel edges unless stated otherwise;
if $G$ is a graph, we use $V(G)$ to denote the vertex set of $G$.
A \emph{triangulation} of the plane or a surface of Euler genus $g>0$
is a graph embedded in such a surface such that every face is a $2$-cell, i.e., homeomorphic to a disk, and
bounded by a triangle.
A \emph{near-triangulation} is a $2$-connected graph $G$ embedded in the plane such that
each inner face of $G$ is bounded by a triangle.

We next give a formal definition of twin-width.
A \emph{trigraph} is a graph with some of its edges being red, and
the \emph{red degree} of a vertex $v$ is the number of red edges incident with $v$.
If $G$ is a trigraph and $v$ and $v'$ form a pair of its (not necessarily adjacent) vertices,
then the trigraph obtained from $G$ by \emph{contracting} the vertices $v$ and $v'$
is the trigraph obtained from $G$ by removing the vertices $v$ and $v'$ and introducing a new vertex $w$ such that
$w$ is adjacent to every vertex $u$ that is adjacent to at least one of the vertices $v$ and $v'$ in $G$ and
the edge $wu$ is red if $u$ is not adjacent to both $v$ and $v'$ or at least one of the edges $vu$ and $v'u$ is red,
i.e., the edge $wu$ is not red only if $G$ contains both edges $vu$ and $v'u$ and neither of the two edges is red.
The \emph{twin-width} of a graph $G$ is the smallest integer $k$ such that
there exists a sequence of contractions that reduces the graph $G$,
i.e., the trigraph with the same vertices and edges as $G$ and no red edges, to a single vertex and
none of the intermediate graphs contains a vertex of red degree more than $k$.

A \emph{rooted tree-decomposition} $\TT$ of a graph $G$ is a rooted tree such that
each vertex of $\TT$ is a subset of $V(G)$, which we refer to as a \emph{bag}, and that
satisfies the following:
\begin{itemize}
\item for every vertex $v$ of $G$, there exists a bag containing $v$,
\item for every vertex $v$ of $G$, the bags containing $v$ form a connected subgraph (subtree) of $\TT$, and
\item for every edge $e$ of $G$, there exists a bag containing both end vertices of $e$.
\end{itemize}
If the choice of the root is not important,
we just speak about a \emph{tree-decomposition} of a graph $G$.
The \emph{width} of a tree-decomposition $\TT$ is the maximum size of a bag of $\TT$ decreased by one, and
the \emph{tree-width} of a graph $G$ is the minimum width of a tree-decomposition of $G$.

A \emph{$k$-tree} is defined recursively as follows:
the complete graph $K_k$ is a $k$-tree and
if $G$ is a $k$-tree,
then any graph obtained from $G$ by introducing a new vertex and making it adjacent to any $k$ vertices of $G$ that
form a complete subgraph in $G$ is also a $k$-tree.
Note that a graph $G$ is a $1$-tree if and only if $G$ is a tree.
More generally, a graph $G$ has tree-width at most $k$ if and only if $G$ is a subgraph of a $k$-tree, and
if $G$ has at least $k$ vertices, then $G$ is actually a spanning subgraph of a $k$-tree.
Note that $k$-trees have a tree-like structure given by their recursive definition,
which also gives a rooted tree-decomposition of $G$ with width $k$:
the rooted tree-decomposition of $K_k$ consists of a single bag containing all $k$ vertices, and
the rooted tree-decomposition of the graph obtained from a $k$-tree $G$ by introducing a vertex $w$
can be obtained from the rooted tree-decomposition $\TT_G$ of $G$ by introducing a new bag containing $w$ and its $k$ neighbors and
making this bag adjacent to the bag of $\TT_G$ that contains all $k$ neighbors of $w$ (such a bag exists since
the subtrees of a tree have the Helly property).

A \emph{BFS spanning tree} $T$ of a (connected) graph $G$ is a rooted spanning tree such that
the path from the root to any vertex $v$ in $T$ is the shortest path from the root to $v$ in $G$;
in particular, a BFS spanning tree can be obtained by the breadth-first search (BFS) of a graph.
A \emph{layering} is a partition of a vertex set of a graph $G$ into sets $V_1,\ldots,V_k$,
which are called \emph{layers}, such that every edge of $G$ connects two vertices of the same or adjacent layers,
i.e., layers whose indices differ by one.
If $T$ is a BFS spanning tree of $G$,
then the partition of the vertex set $V(G)$ into sets based on the distance from the root of $T$,
i.e., the first set contains the root,
the second set contains all neighbors of the root,
the third set contains all vertices at distance two from the root, etc.,
is a layering.
A \emph{BFS spanning forest} $F$ is a rooted spanning forest of $G$,
i.e., a forest consisting of rooted trees, such that
there exists a layering $V_1,\ldots,V_k$ of $G$ compatible with $F$,
i.e., for every tree of $F$,
there exists $d$ such that the vertices at distance $\ell$ from the root are all contained in $V_{d+\ell}$.
Note that if $G$ is a graph and $T$ a BFS spanning tree of $G$,
then removing the same vertices in $G$ and $T$, which naturally yields a rooted forest,
results in a graph $G'$ and a BFS spanning forest of $G'$.
Finally, the \emph{layered tree-width} of a graph $G$ is the minimum $k$ for which
there exists a tree-decomposition $\TT$ of $G$ and a layering such that
every bag of $\TT$ contains at most $k$ vertices from the same layer.

Consider a graph $G$ and a BFS spanning tree $T$ of $G$.
A \emph{vertical} path is a path contained in $T$ with no two vertices from the same layer,
i.e., a subpath of a path from a leaf to the root of $T$.
The \emph{top} vertex of a vertical path is its vertex closest to the root and
the \emph{bottom} vertex is its vertex farthest from the root.
We define vertical paths with respect to a BFS spanning forest analogously.
If $\PP$ is a partition of the vertex set of $G$ to vertical paths,
the graph $G/\PP$ is the graph obtained by contracting each of the paths contained in $\PP$ to a single vertex;
note that the vertices of $G/\PP$ can be viewed as the vertical paths contained in $\PP$ and
two vertical paths $P$ and $P'$ are adjacent in $G/\PP$ if there is an edge between $V(P)$ and $V(P')$,
i.e., there is a vertex of $P$ adjacent to a vertex of $P'$.

\section{Product Structure Theorem for graphs on surfaces}
\label{sec:product}

In this section, we provide the version of the Product Structure Theorem for graphs on surfaces,
which we need to prove our upper bound on the twin-width of such graphs.
We start with recalling the following lemma proven by Dujmovi\'c et al.~\cite{DujJMMUW19,DujJMMUW20};
also see Figure~\ref{fig:BFS-torus} for the illustration in the case of the torus.

\begin{lemma}
\label{lm:BFS}
Let $G$ be a triangulation of a surface $\Sigma$ of Euler genus $g>0$ and let $T$ be a BFS spanning tree of $G$.
There exist edges $a_1b_1,\ldots,a_gb_g$ not contained in the tree $T$ with the following property.
Let $F_0$ be the subset of edges of $G$ comprised of the $g$ edges $a_1b_1,\ldots,a_gb_g$ and
the $2g$ paths in $T$ from the root of $T$ to the vertices $a_1,\ldots,a_g$ and $b_1,\ldots,b_g$.
The surface $\Sigma$ after the removal of the edges contained in $F_0$ is homeomorphic to a disk and
its boundary is formed by a closed walk comprised of the edges contained in $F_0$.
\end{lemma}

Using Lemma~\ref{lm:BFS}, we prove the following.

\begin{lemma}
\label{lm:BFS-paths}
Let $G$ be a triangulation of a surface of Euler genus $g>0$ and let $T$ be a BFS spanning tree of $G$.
There exist a closed walk $W$ in $G$,
a subtree $T_0$ of $T$ that contains the root of $T$, and
$k$ vertex-disjoint vertical paths $P_1,\ldots,P_k$, $k\le 2g$, such that
\begin{itemize}
\item the closed walk $W$ bounds a part of the surface homeomorphic to a disk,
\item the sets $V(P_1),\ldots,V(P_k)$ form a partition of $V(T_0)$, 
      i.e., $V(T_0)=V(P_1)\cup\cdots\cup V(P_k)$, and
\item the closed walk $W$ can be split into $6g-1$ segments,
      each formed by vertices of exactly one of the paths $P_1,\ldots,P_k$.
\end{itemize}
\end{lemma}

\begin{proof}
Fix a triangulation $G$ and a BFS spanning tree $T$.
Apply Lemma~\ref{lm:BFS} to get edges $a_1b_1,\ldots,a_gb_g$.
The tree $T_0$ is the subtree of $T$ formed
by the paths from the root to the vertices $a_1,\ldots,a_g$ and $b_1,\ldots,b_g$.
Let $W$ be the unique closed walk comprised of the edges of $T_0$ and the edges $a_1b_1,\ldots,a_gb_g$ that
forms the boundary of the surface homeomorphic to a disk obtained by removing the edges of $T_0$.
We refer to Figure~\ref{fig:BFS-torus} for illustration in the case of the torus.

\begin{figure}
\begin{center}
\epsfbox{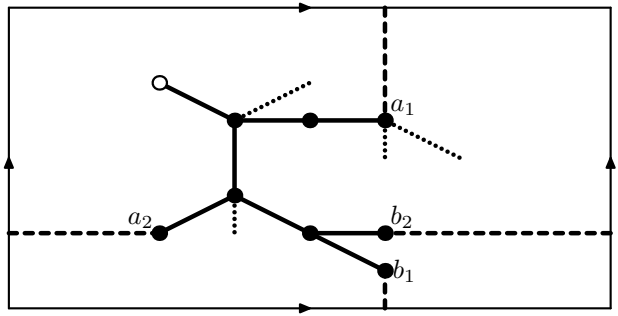}
\end{center}
\caption{A rooted tree $T_0$ and edges $a_1b_1$ and $a_2b_2$, which are drawn dashed,
         bounding a part of the torus that is homeomorphic to a disk
	 as in the proof of Lemma~\ref{lm:BFS-paths}.
	 Possible additional edges of the BFS spanning tree $T$ are drawn dotted.
         The root of the tree $T_0$ is depicted by an empty circle unlike the other vertices of $T_0$.}
\label{fig:BFS-torus}
\end{figure}

We next define paths $P_1,\ldots,P_k$.
The embedding of the tree $T_0$ in the surface induces a natural cyclic ordering at each vertex (regardless
whether the surface is orientable or not).
Hence, it is possible to order the leaves of the tree $T_0$ from left to right and
let $v_1,\ldots,v_k$ be the leaves of $T_0$ listed in this order.
Note that $k\le 2g$ since each of the leaves is one of the vertices $a_1,\ldots,a_g$ and $b_1,\ldots,b_g$.
The path $P_1$ is the path contained in $T_0$ from the leaf $v_1$ to the root of the tree $T_0$, and
the path $P_i$, $i=2,\ldots,k$,
is the path contained in $T_0$ from the leaf $v_k$ to the child of a vertex in $V(P_1)\cup\cdots\cup V(P_{i-1})$.
Note that some of the paths $P_2,\ldots,P_k$ may consist of the leaf vertex only.
We refer to Figure~\ref{fig:BFS-tree} for illustration.

\begin{figure}
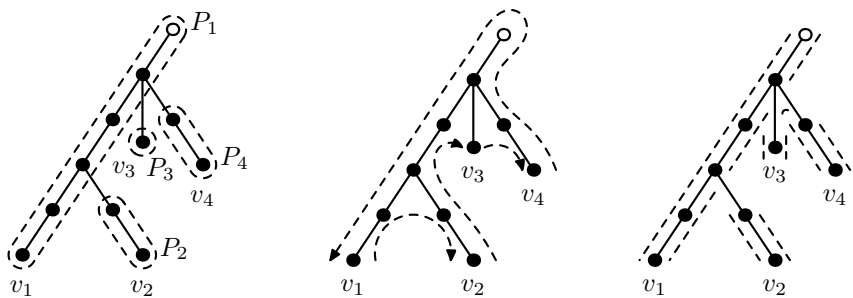

\begin{center}
\epsfbox{twing-1.mps}
\hskip 5ex
\epsfbox{twing-2.mps}
\hskip 5ex
\epsfbox{twing-3.mps}
\end{center}
\caption{The decomposition of the tree $T_0$ into paths $P_1,\ldots,P_4$ (the case $g=2$ and $k=4$),
         the illustration of the split into $2g=4$ segments covering the walk $W$ as in the proof of Lemma~\ref{lm:BFS-paths}, and
	 the collection of $11$ vertical paths covering $W$ obtained in the proof.
	 The root of the tree $T_0$ is depicted by an empty circle unlike the other vertices.}
\label{fig:BFS-tree}
\end{figure}

It remains to prove the last claim of the statement of the lemma.
We show that the closed walk $W$ can be split into $2g+2k-1$ segments,
each formed by vertices of exactly one of the paths $P_1,\ldots,P_k$.
Observe that the closed walk $W$ can be split into $2g$ segments,
each delimited by a pair of the vertices $a_1,\ldots,a_g$ and $b_1,\ldots,b_g$, and
each of these $2g$ segments
is a walk in the tree $T_0$ between the corresponding pair of vertices $a_1,\ldots,a_g$ and $b_1,\ldots,b_g$ that
is consistent with the cyclic order of edges at the vertices of $T_0$.
Figure~\ref{fig:BFS-tree} gives an illustration for the case $g=2$ and $k=4$,
i.e., all four vertices $a_1$, $a_2$, $b_1$ and $b_2$ are the leaves of $T_0$.

We first analyze the case $k=2g$, i.e., all segments are delimited by the leaves of the tree $T_0$.
Let us think of the tree $T_0$ as obtained by adding paths $P_1,\ldots,P_k$ sequentially one after another.
At the beginning, the closed walk around the path $P_1$ can be covered by two segments:
one following $P_1$ from $v_1$ to the root on the left and the other on the right.
When the path $P_i$ is added, we may think of it as adding two new segments following $P_i$ on the left and on the right and
splitting the segment containing the parent of the top vertex of $P_i$ into two segments that
overlap at the parent of the top vertex of the path $P_i$ (see Figure~\ref{fig:BFS-tree}).
Hence, three segments are added to the collection at each step and
so the number of segments needed to cover $W$ is $3k-1$.

Suppose now that some of the vertices $a_1,\ldots,a_g$ and $b_1,\ldots,b_g$ are not among the leaves of $T_0$ or
some of the leaves of $T_0$ correspond to multiple vertices among $a_1,\ldots,a_g$ and $b_1,\ldots,b_g$.
Observe that there are exactly $2g-k$ such vertices (when counting multiplicities) and
each of them splits one of the created segments into two.
Hence, the total number of segments is $3k-1+(2g-k)=2g+2k-1$,
which implies the bound in the statement of the lemma as $k\le 2g$.
\end{proof}

Lemma~\ref{lm:BFS-paths} forms one of the two key ingredients for the proof of Theorem~\ref{thm:product}.
The second relates to partitioning disk regions bounded by vertical paths.
Similarly to~\cite{DujJMMUW19,DujJMMUW20,UecWY22},
we make use of Sperner's Lemma, see e.g.~\cite{AigZ10,Spe80}.

\begin{lemma}
\label{lm:sperner}
Let $G$ be a near-triangulation.
Suppose that the vertices of $G$ are colored with three colors in such a way that
the vertices of each of the three colors on the outer face are consecutive,
i.e., they form a non-empty path.
There exists an inner face that contains one vertex of each of the three colors.
\end{lemma}

The proof of the next lemma follows the lines of the proof of~\cite[Lemma 8]{UecWY22};
we include a proof for completeness.

\begin{lemma}
\label{lm:tw7}
Let $G$ be a near-triangulation and
let $T$ be a BFS rooted spanning forest such that all roots of $T$ are on the outer face.
If the boundary cycle of the outer face can be partitioned into at most $6$ vertical paths,
say $P_1,\ldots,P_k$, $k\le 6$,
then there exists a collection $\PP$ of vertex-disjoint vertical paths such that
\begin{itemize}
\item the collection $\PP$ contains the paths $P_1,\ldots,P_k$,
\item every vertex of $G$ is contained in one of the paths in $\PP$, and
\item $G/\PP$ has a rooted tree-decomposition of width at most seven such that
      the root bag contains the vertices corresponding to the paths $P_1,\ldots,P_k$.
\end{itemize}
\end{lemma}

\begin{proof}
We proceed by induction on the number of inner vertices of $G$.
The base case is when $G$ has no inner vertices.
In this case, we simply set $\PP$ to be the set $\{P_1,\ldots,P_k\}$.
We next present the induction step.

Fix a near-triangulation $G$ and the paths $P_1,\ldots,P_k$ as described in the statement of the lemma.
We will say that an inner vertex $v$ of $G$ is \emph{reachable} from a path $P_i$
if the tree of $T$ that contains $v$ also contains $P_i$ and
the path $P_i$ is the first among the paths $P_1,\ldots,P_k$ hit
on the path from $v$ to the root of the tree of $T$ that contains $v$.

If $k\le 5$, we proceed as follows.
Color the vertices of the path $P_1$ and those reachable from this path through $T$ red,
the vertices of the path $P_2$ and those reachable from this path through $T$ green,
the vertices of the paths $P_3,\ldots,P_k$ and those reachable from these $k-2$ paths through $T$ blue.
By Lemma~\ref{lm:sperner}, there exists an inner face with a red vertex, a green vertex and a blue vertex.
Let $A$, $B$ and $C$ be the vertical paths from the path $P_1$, the path $P_2$ and a path $P_i$, $i\in\{3,\ldots,k\}$ to these three vertices.
The case $k=5$ is illustrated in Figure~\ref{fig:tw7a}.
We now apply induction to the inner triangulation delimited by parts of the paths $P_1$ and $P_2$ and the paths $A$ and $B$,
the inner triangulation delimited by parts of the paths $P_2$ and $P_i$ and the paths $P_3,\ldots,P_{i-1},B$ and $C$, and
the inner triangulation delimited by parts of the paths $P_1$ and $P_i$ and the paths $P_{i+1},\ldots,P_5,A$ and $C$.
The collection $\PP$ is the union of the three collections produced by induction
with parts of paths $P_1$, $P_2$ and $P_i$ replaced by the whole paths $P_1$, $P_2$ and $P_i$.
The rooted tree-decomposition of $G/\PP$ is obtained by creating a new root bag containing the $k+3\le 8$ paths $P_1,\ldots,P_k$ and $A,B,C$, and
making the root bags of the three tree-decompositions obtained by induction to be adjacent to this root bag.

\begin{figure}
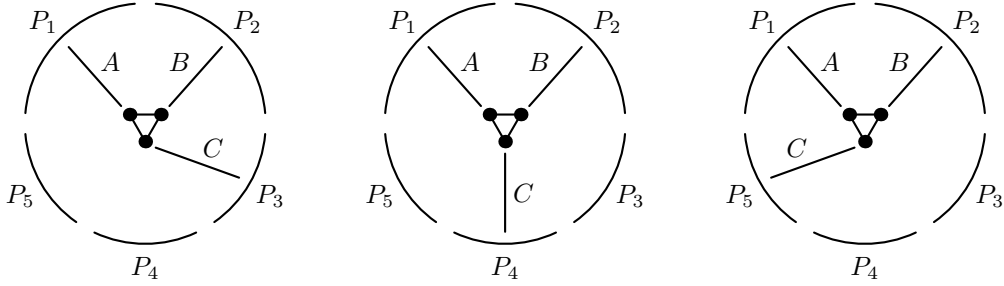

\begin{center}
\epsfbox{twing-5.mps}
\hskip 5ex
\epsfbox{twing-6.mps}
\hskip 5ex
\epsfbox{twing-7.mps}
\end{center}
\caption{The three possible cases where the path $C$ can lead to
         in the induction step in the proof of Lemma~\ref{lm:tw7} for $k=5$.}
\label{fig:tw7a}
\end{figure}

We now deal with the case $k=6$.
Color the vertices of the paths $P_1$ and $P_2$ and those reachable from these two paths through $T$ red,
the vertices of the paths $P_3$ and $P_4$ and those reachable from these two paths through $T$ green,
the vertices of the paths $P_5$ and $P_6$ and those reachable from these two paths through $T$ blue.
Let $A$, $B$ and $C$ be the vertical paths from the paths $P_1,\ldots,P_k$ to these three vertices, and
let $a$, $b$ and $c$ be the indices such that
the path $A$ starts at a vertex adjacent to the path $P_a$,
$B$ starts at a vertex adjacent to the path $P_b$, and
$C$ starts at a vertex adjacent to the path $P_c$.
Note that at least one of the following two cases holds: $b-a\ge 2$ or $c-b\ge 2$.
By symmetry, we assume that $b-a\ge 2$ in the rest.
The setting is also illustrated in Figure~\ref{fig:tw7b}.

\begin{figure}
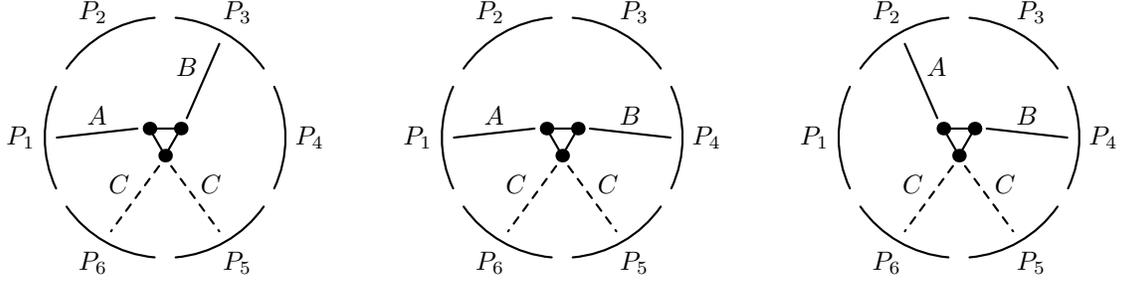

\begin{center}
\epsfbox{twing-8.mps}
\hskip 5ex
\epsfbox{twing-9.mps}
\hskip 5ex
\epsfbox{twing-10.mps}
\end{center}
\caption{The three possible cases of the values $a$ and $b$ in the induction step in the proof of Lemma~\ref{lm:tw7} for $k=6$
         under the assumption $b-a\ge 2$.}
\label{fig:tw7b}
\end{figure}

We now apply induction to the inner triangulation $G_{ab}$ delimited by parts of the paths $P_a$ and $P_b$, the paths $P_{a+1},\ldots,P_{b-1},A$ and $B$,
the inner triangulation $G_{bc}$ delimited by parts of the paths $P_b$ and $P_c$, the paths $P_{b+1},\ldots,P_{c-1},B$ and $C$, and
the inner triangulation $G_{ca}$ delimited by parts of the paths $P_c$ and $P_a$, the paths $P_{c+1},\ldots,P_{a-1},C$ and $A$.
The collection $\PP$ is the union of the three collections produced by induction
with parts of paths $P_a$, $P_b$ and $P_c$ replaced by the whole paths $P_a$, $P_b$ and $P_c$.

The sought rooted tree-decomposition of $G/\PP$ is obtained as follows.
First create a root bag containing the eight paths $P_1,\ldots,P_k$ and $A$ and $B$.
The root will have two children.
One of the children is the root bag of the rooted tree-decomposition of $G_{ab}$ obtained by induction and
the rooted tree-decomposition of $G_{ab}$ forms a subtree rooted at this bag;
note that this bag contains the paths $P_a,\ldots,P_b$ and $A$ and $B$.
The other child is a new bag containing the (at most eight) paths $P_b,\ldots,P_a$ and $A$, $B$ and $C$;
this bag will have two children.
One of them is the root bag of the rooted tree-decomposition of $G_{bc}$ and
the other is the root bag of the rooted tree-decomposition of $G_{ca}$,
which both have been obtained by induction;
the rooted tree-decompositions of $G_{bc}$ and $G_{ca}$ form subtrees rooted at these two bags.
As the obtained rooted tree-decomposition of $G/\PP$ has width at most seven and
has the properties given in the statement of the lemma,
the proof of the lemma is finished.
\end{proof}

The next lemma is the last ingredient needed to prove the main result of this section,
which is Theorem~\ref{thm:product}.

\begin{lemma}
\label{lm:twout}
Let $G$ be a near-triangulation and
let $T$ be a BFS rooted spanning forest such that all roots of $T$ are on the outer face.
If the boundary cycle of the outer face can be partitioned into $k\ge 6$ vertical paths,
then there exist a $2$-connected subgraph $G'$ of $G$ and a collection $\PP$ of vertex-disjoint vertical paths such that
\begin{itemize}
\item $\PP$ contains all vertical paths bounding the outer face,
\item $\PP$ contains at most $\max\{6,6k-32\}$ paths,
\item the vertex set of $G'$ is the union of the vertex sets of the paths contained in $\PP$,
\item the graph $G'$ contains the boundary of the outer face,
\item the graph $G'$ has at most $\max\{1,3k-18\}$ inner faces, and
\item each of the inner faces of $G'$ is bounded by at most six paths contained in $\PP$.
\end{itemize}
\end{lemma}

\begin{proof}
The proof proceeds by induction on $k$.
The $k$ vertical paths bounding the outer face are denoted by $P_1,\ldots,P_k$.
We distinguish four cases depending on the value of $k$: $k=6$, $k=7$, $k=8$ and $k\ge 9$.
If $k=6$, we just set $G'$ to be the cycle bounding the outer face and $\PP=\{P_1,\ldots,P_6\}$.

We next analyze the case $k=7$.
Color the vertices of the paths $P_1$, $P_2$ and $P_3$ and those reachable from these three paths through $T$ red,
the vertices of the paths $P_4$ and $P_5$ and those reachable from these two paths through $T$ green, and
the remaining vertices,
i.e., the vertices of the paths $P_6$ and $P_7$ and those reachable from these two paths through $T$, blue.
The meaning of being \emph{reachable} as in the proof of Lemma~\ref{lm:tw7}.
By Lemma~\ref{lm:sperner}, there exists an inner face with a red vertex, a green vertex and a blue vertex.
Let $A$, $B$ and $C$ be the vertical paths from the paths $P_1,\ldots,P_k$ to these three vertices, and
let $a$, $b$ and $c$ be the indices such that
the path $A$ starts at a vertex adjacent to the path $P_a$,
$B$ starts at a vertex adjacent to the path $P_b$, and
$C$ starts at a vertex adjacent to the path $P_c$.
The subgraph $G'$ formed by the vertices of the paths $P_1,\ldots,P_7$ and the paths $A,B,C$, and
the set $\PP=\{P_1,\ldots,P_7,A,B,C\}$ (note that $|\PP|=10$) satisfies the statement of the lemma
unless $(a,b)=(1,5)$ or $(a,c)=(3,6)$ (see Figure~\ref{fig:twout1}).
As the two cases are symmetric, we analyze the former only.

\begin{figure}
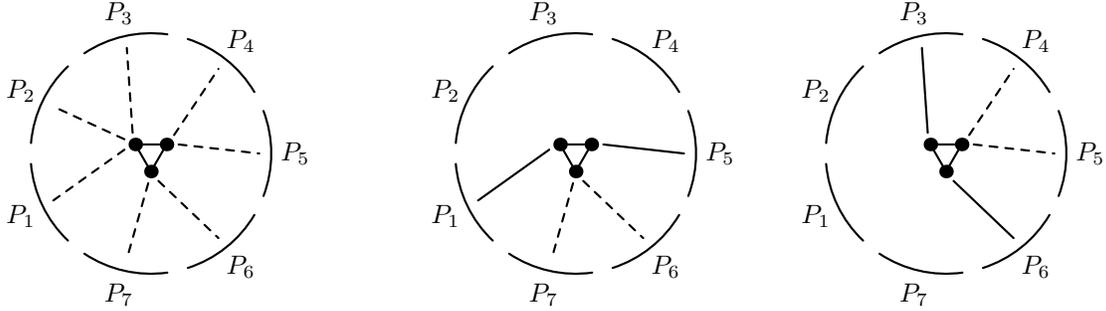

\begin{center}
\epsfbox{twing-11.mps}
\hskip 8ex
\epsfbox{twing-12.mps}
\hskip 4ex
\epsfbox{twing-13.mps}
\end{center}
\caption{The initial step of the case $k=7$ in the proof of Lemma~\ref{lm:twout};
         possible paths $A$, $B$ and $C$ are drawn dashed.
         The two cases when a further argument is needed are depicted in the right.}
\label{fig:twout1}
\end{figure}

\begin{figure}
\begin{center}
\epsfbox{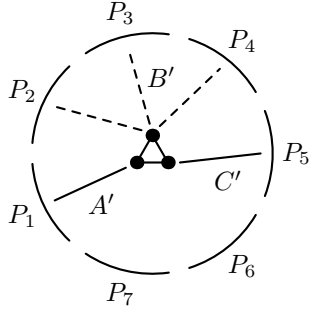}
\end{center}
\caption{The paths $A'$, $B'$ and $C'$ obtained in the second step of the case $k=7$ in the proof of Lemma~\ref{lm:twout}
         assuming $a=1$ and $b=5$.}
\label{fig:twout2}
\end{figure}

Color the vertices of the paths $P_2,\ldots,P_4$ and those reachable from these three paths through $T$ yellow,
the vertices of the path $P_1$ and those reachable from this path through $T$ magenta, and
the vertices of the path $P_5$ and those reachable from this path through $T$ cyan.
We apply Lemma~\ref{lm:sperner} to the subgraph of $G$ bounded by the paths $P_1,\ldots,P_5$ and
the paths $A$ and $B$ (note that all vertices of $A$ are colored with magenta and all vertices of $B$ with cyan).
This subgraph contains an inner face with a magenta vertex, a cyan vertex and a yellow vertex, and
let $A'$, $B'$ and $C'$ be the vertical paths leading from a vertex of $V(P_1)$,
a vertex of $V(P_5)$ and a vertex of $V(P_2)\cup V(P_3)\cup V(P_4)$ to these three vertices;
see Figure~\ref{fig:twout2}.
Since the subgraph $G'$ formed by the vertices of the paths $P_1,\ldots,P_7$ and the paths $A',B',C'$, and
the set $\PP=\{P_1,\ldots,P_7,A',B',C'\}$ (note that $|\PP|=10=6\cdot 7-32$) satisfies the statement of the lemma,
the analysis of the case $k=7$ is finished.

\begin{figure}
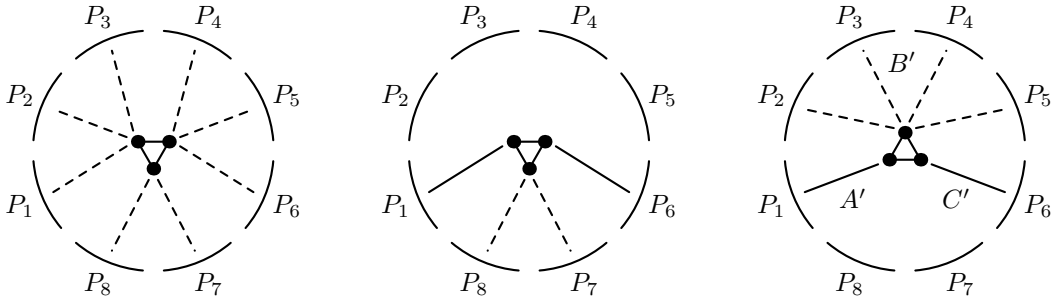

\begin{center}
\epsfbox{twing-15.mps}
\hskip 5ex
\epsfbox{twing-16.mps}
\hskip 5ex
\epsfbox{twing-17.mps}
\end{center}
\caption{The initial step of the case $k=8$ in the proof of Lemma~\ref{lm:twout};
         possible paths $A$, $B$ and $C$ are drawn dashed.
	 The case $(a,b)=(1,6)$ requiring a special argument is depicted in the middle, and
	 the paths obtained in this case in the right.}
\label{fig:twout3}
\end{figure}

We next analyze the case $k=8$.
Color the vertices of the paths $P_1,\ldots,P_3$ and those reachable from these paths through $T$ red,
the vertices of the paths $P_4,\ldots,P_6$ and those reachable from these paths through $T$ green, and
the remaining vertices,
i.e., the vertices of the paths $P_7$ and $P_8$ and those reachable from these paths through $T$, blue.
By Lemma~\ref{lm:sperner}, there exists an inner face with a red vertex, a green vertex and a blue vertex.
Let $A$, $B$ and $C$ be the vertical paths from the paths $P_1,\ldots,P_k$ to these three vertices, and
let $a$, $b$ and $c$ be the indices such that
the path $A$ starts at a vertex adjacent to the path $P_a$,
$B$ starts at a vertex adjacent to the path $P_b$, and
$C$ starts at a vertex adjacent to the path $P_c$.
See Figure~\ref{fig:twout3}.
Unless $a=1$ and $b=6$, we proceed as follows.
If each of the three near-triangulations delimited by the paths $P_1,\ldots,P_8$ and the paths $A$, $B$ and $C$
is bounded by at most six of these paths,
we set $G'$ to be the subgraph formed by the vertices of the paths $P_1,\ldots,P_8$ and the paths $A,B,C$, and
$\PP={P_1,\ldots,P_8,A,B,C}$.
Since $|\PP|=11<16=6\cdot 8-32$ and $G'$ has three inner faces, the statement of the lemma holds.
Otherwise, exactly one of the three near-triangulations delimited by the paths $P_1,\ldots,P_8$ and the paths $A$, $B$ and $C$
is bounded by seven of these paths and the remaining two by at most six of these paths.
We apply induction to the near-triangulation delimited by seven of the paths $P_1,\ldots,P_8$ and the paths $A$, $B$ and $C$, and
we set $G'$ to be the subgraph formed by the vertices of the paths $P_1,\ldots,P_8$, the paths $A,B,C$ and the subgraph obtained by induction, and
the set $\PP$ to contain the paths $P_1,\ldots,P_8$, the paths $A$, $B$ and $C$, and
the at most three additional paths obtained by induction.
Note that $G'$ has at most five inner faces and the set $\PP$ contains at most $11+3=14\le 16$ paths.

We next assume that $a=1$ and $b=6$.
Color the vertices of the paths $P_2,\ldots,P_4$ and those reachable from these three paths through $T$ yellow,
the vertices of the path $P_1$ and those reachable from this path through $T$ magenta, and
the vertices of the path $P_6$ and those reachable from this path through $T$ cyan.
Note that all vertices of $A$ are colored with magenta and all vertices of $B$ with cyan.
We apply Lemma~\ref{lm:sperner} to the subgraph of $G$ bounded by the paths $P_1,\ldots,P_5$ and
the paths $A$ and $B$ (see Figure~\ref{fig:twout3}).
This subgraph contains an inner face with a magenta vertex, a cyan vertex and a yellow vertex, and
let $A'$, $B'$ and $C'$ be the vertical paths leading from a vertex of $V(P_1)$,
a vertex $V(P_6)$ and a vertex of $V(P_2)\cup\cdots\cup V(P_5)$ to these three vertices.
If the path $C'$ leads to a vertex of $V(P_3)\cup V(P_4)$,
then we set $G'$ to be the subgraph formed by the vertices of the paths $P_1,\ldots,P_8$ and the paths $A'$, $B'$ and $C'$ and
the set $\PP$ to be the set $\{P_1,\ldots,P_8,A',B',C'\}$.
Since $G'$ has three inner faces and $\PP$ contains $11$ paths,
the statement of the lemma follows.
If the path $C'$ leads to a vertex of $V(P_2)\cup V(P_5)$, we may assume that it leads to a vertex of $V(P_5)$ by symmetry.
We apply induction to the near-triangulation delimited by the paths $P_1,\ldots,P_5$ and the paths $A'$ and $C'$, and
we set $G'$ to be the subgraph formed by the vertices of the paths $P_1,\ldots,P_8$, the paths $A',B',C'$ and the subgraph obtained by induction, and
the set $\PP$ to contain the paths $P_1,\ldots,P_8$, the paths $A'$, $B'$ and $C'$ and
the (at most three) additional paths obtained by induction.
Since $G'$ has at most five inner faces and $\PP$ contains at most $11+3=14\le 16$ paths,
the statement of the lemma follows.

\begin{figure}
\begin{center}
\epsfbox{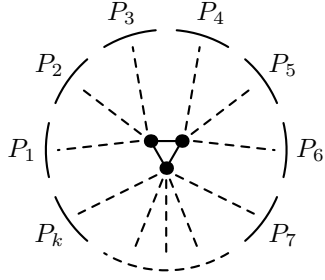}
\end{center}
\caption{The initial step of the case $k\ge 9$ in the proof of Lemma~\ref{lm:twout};
         possible paths $A$, $B$ and $C$ are drawn dashed.}
\label{fig:twout4}
\end{figure}

We now analyze the general case $k\ge 9$.
Color the vertices of the paths $P_1,\ldots,P_3$ and those reachable from these paths through $T$ red,
the vertices of the paths $P_4,\ldots,P_6$ and those reachable from these paths through $T$ green, and
the remaining vertices,
i.e., the vertices of the paths $P_7,\ldots,P_k$ and those reachable from these paths through $T$, blue.
By Lemma~\ref{lm:sperner}, there exists an inner face with a red vertex, a green vertex and a blue vertex.
Let $A$, $B$ and $C$ be the vertical paths from the paths $P_1,\ldots,P_k$ to these three vertices, and
let $a$, $b$ and $c$ be the indices such that
the path $A$ starts at a vertex adjacent to the path $P_a$,
$B$ starts at a vertex adjacent to the path $P_b$, and
$C$ starts at a vertex adjacent to the path $P_c$.
See Figure~\ref{fig:twout4}.
Let $\ell_{ab}=b-a$, $\ell_{bc}=c-b$ and $\ell_{ca}=k+a-c$.
Since $\ell_{ab}+\ell_{bc}+\ell_{ca}=k\ge 9$,
one of $\ell_{ab}$, $\ell_{bc}$ and $\ell_{ca}$ is at least four
unless $\ell_{ab}=\ell_{bc}=\ell_{ca}=3$ and so $k=9$.
If $\ell_{ab}=\ell_{bc}=\ell_{ca}=3$,
we set $G'$ to be the subgraph formed by the vertices of the paths $P_1,\ldots,P_9$ and the paths $A,B,C$, and
$\PP={P_1,\ldots,P_9,A,B,C}$.
Since $|\PP|=12$ and $G'$ has three inner faces, each bounded by six of the paths contained in $\PP$,
the statement of the lemma holds.

In the rest of the proof,
we assume that at least one of $\ell_{ab}$, $\ell_{bc}$ and $\ell_{ca}$ is at least four.
If $\ell_{ab}\ge 4$,
we apply induction to the near-triangulation bounded by the paths $P_a,\ldots,P_b$ and the paths $A$ and $B$, and
to the near-triangulation bounded by paths $P_b,\ldots,P_k,P_1,\ldots,P_a$ and the paths $A$ and $B$.
If $\ell_{ab}<4$ but $\ell_{bc}\ge 4$,
we apply induction to the near-triangulation bounded by the paths $P_b,\ldots,P_c$ and the paths $B$ and $C$, and
to the near-triangulation bounded by paths $P_c,\ldots,P_k,P_1,\ldots,P_b$ and the paths $B$ and $C$.
Otherwise,
we apply induction to the near-triangulation bounded by the paths $P_a,\ldots,P_c$ and the paths $A$ and $C$, and
to the near-triangulation bounded by paths $P_c,\ldots,P_k,P_1,\ldots,P_a$ and the paths $B$ and $C$.
In each of these three cases,
we set the sought graph $G'$ to be the union of the two graphs obtained by induction and
$\PP$ to be the union of the obtained sets of vertical paths.

It remains to estimate the size of $\PP$ and the number of faces of $G'$.
Let $k_1$ be the number of paths bounding the former of the two near-triangulations to that the induction has been applied, and
let $k_2$ be the number of paths bounding the latter.
Observe that $k_1+k_2=k+6$ and both $k_1$ and $k_2$ are at least seven.
By induction,
the set $\PP$ contains at most $k+2+(5k_1-32)+(5k_2-32)=6k+2+30-2\cdot 32=6k-32$ vertical paths and
the number of inner faces of $G'$ is at most $(3k_1-18)+(3k_2-18)=3k+18-36=3k-18$.
The proof of the lemma is now completed.
\end{proof}

We are now ready to prove the main result of this section, which implies Theorem~\ref{thm:prodnew}.

\begin{theorem}
\label{thm:product}
Let $G$ be a triangulation of a surface of Euler genus $g>0$ and let $T$ be a BFS spanning tree of $G$.
The tree $T$ can be vertex-partitioned to a collection $\PP$ of vertex-disjoint vertical paths such that
the graph $G/\PP$ has a rooted tree-decomposition with the following properties:
\begin{itemize}
\item the root bag has size at most $\max\{6,32g-27\}$,
\item the root bag has at most $6\cdot\max\{1,18g-21\}$ children, and
\item every bag except the root bag has size at most $8$.
\end{itemize}
Moreover, every subtree $\TT'$ of the tree-decomposition rooted at a child of the root satisfies the following:
\begin{itemize}
\item the bags of $\TT'$ contain at most six paths that are contained in the root bag, and
\item if $P_1,\ldots,P_k$ are all paths that are contained in the bags of $\TT'$ but not in the root bag,
      the subgraph induced by $V(P_1)\cup\cdots\cup V(P_k)$
      has a component joined by an edge to each of the paths that are contained both in the root bag and in $\TT'$.
\end{itemize}      
\end{theorem}

\begin{proof}
Fix a triangulation $G$ of a surface of Euler genus $g>0$ and a BFS spanning tree $T$ of $G$.
We apply Lemma~\ref{lm:BFS-paths} to obtain a closed walk $W$, a subtree $T_0$ of $T$ and
$k$ vertex-disjoint vertical paths $P_1,\ldots,P_k$, $k\le 2g$, with the properties given in Lemma~\ref{lm:BFS-paths}.
Let $\ell$ be the number segments that cover the walk $W$ such that each segment
is formed by vertices of exactly one of the paths $P_1,\ldots,P_k$;
note that $\ell\le 6g-1$ by Lemma~\ref{lm:BFS-paths}.

We first deal with the general case $\ell\ge 7$ (note that if $\ell\ge 7$, then $g\ge 2$).
We apply Lemma~\ref{lm:twout} to the near-triangulation bounded by the closed walk $W$,
the BFS spanning forest obtained from $T$ by duplicating the vertices contained in $W$, and
the $\ell$ vertical paths that corresponds to the segments covering the walk $W$.
We obtain a collection $\PP_0$ of vertex-disjoint vertical paths that contains at most $5\ell-32$ additional vertical paths and
a $2$-connected subgraph $G'$ of $G$ such that each inner face of $G'$ is bounded by at most six paths contained in $\PP_0$.
In addition, the number of inner faces of $G'$, which we denote $f$ further, is at most $3\ell-18$.
Since $\ell\le 6g-1$, we obtain that $\PP_0$ contains at most $30g-27$ additional vertical paths and
the number $f$ of faces of $G'$ is at most $18g-21$.
We now replace in the collection $\PP_0$ the subpaths of $P_1,\ldots,P_k$ that cover the closed walk $W$
with the paths $P_1,\ldots,P_k$.
Hence, the size of the collection $\PP_0$ is at most $32g-27$ (note that $k\le 2g$) and
each of the faces of $G'$ is still bounded by at most six subpaths of the vertical paths contained in $\PP$ (possibly by multiple subpaths of the same vertical path).

We will soon proceed jointly for the cases $\ell\ge 7$ and $\ell\le 6$.
To be able to do so, in the case $\ell\le 6$ (and so $k\le 6$),
we set $\PP_0$ to to be the collection $\{P_1,\ldots,P_k\}$ and
$G'$ the graph consisting of a single cycle corresponding to the closed walk $W$;
note that the only face of $G'$ bounds a near-triangulation in $G$ and $f=1$.

We now proceed jointly for both cases.
If there is a face of $G'$ such that the subgraph of $G$ induced by the vertices of $G$ inside this face of $G'$
does not have a component joined by an edge to each of the (at most six) paths from $\PP_0$ bounding the face,
then two vertices on the boundary of this face of $G'$
must be joined by a chord separating the corresponding parts of the interior of the face.
By repeatedly adding such separating chords to $G'$,
we obtain a $2$-connected graph $G''$ with $f'\le 6f$ non-empty inner faces,
each bounded by at most six paths contained in $\PP_0$, and such that
the subgraph of $G$ induced by vertices of $G$ inside any face of $G''$
has a component joined by an edge to each of the (at most six) paths from $\PP_0$ bounding the face.
We now apply Lemma~\ref{lm:tw7} to each of the $f'$ near-triangulations bounded by the non-empty faces of $G''$ and
obtain rooted tree-decompositions $\TT_1,\ldots,\TT_{f'}$ with width at most seven of each them.
Let $\PP$ be the collection of vertical paths obtained from $\PP_)$
by including all additional vertical paths obtained by these $f'$ applications of Lemma~\ref{lm:tw7}.

We now construct a rooted tree-decomposition of $G/\PP$ with the properties given in the statement of the theorem.
We introduce a new root bag containing the vertices corresponding to the paths contained in $\PP_0$ and
make the roots of the rooted tree-decompositions $\TT_1,\ldots,\TT_{f'}$ its children.
The root bag of the resulting tree-decomposition has size $|\PP_0|\le\max\{6,32g-27\}$,
it has $f'\le 6f\le 6\max\{1,18g-21\}$ children, and all bags except the root has size at most $8$.
Consider now a subtree $\TT'$ rooted at a child of the root.
The only paths from $\PP_0$ contained in the bags of $\TT'$ are those bounding the face of $G''$ that corresponds to $\TT'$,
i.e., there are at most six such paths, and
the vertices contained in the paths of the bags of $\TT'$ but not in the paths of $\PP_0$
are exactly the vertices of $G$ contained inside this face of $G''$.
So, the subgraph of $G$ induced by such vertices has a component joined by an edge to each of the paths from $\PP_0$ bounding the face of $G''$.
We conclude that the obtained rooted tree-decompositions has the properties given in the statement of the theorem.
\end{proof}

\section{Upper bound}
\label{sec:upper}

We now present the asymptotically optimal upper bound on the twin-width of graphs embeddable in surfaces.

\begin{theorem}
\label{thm:main}
The twin-width of every graph $G$ of Euler genus $g\ge 1$ is at most
\[6\,\cdot\,\max\left\{3\sqrt{47g}+1,2^{24}\right\}=18\sqrt{47g}+O(1).\]
\end{theorem}

\begin{proof}
Fix a graph $G$ of Euler genus $g>0$ and
let $G_0$ be any triangulation of the surface with Euler genus $g$ that
$G$ is a spanning subgraph of $G_0$, i.e., $V(G_0)=V(G)$ (to avoid
unnecessary technical issues related to adding new vertices, we permit $G_0$ to contain parallel edges).

We apply Theorem~\ref{thm:product} to the triangulation $G_0$ and an arbitrary BFS spanning tree $T_0$, and
let $\PP$ be a collection of vertical paths and a rooted tree-decomposition with the properties given in the theorem.
Let $P_1,\ldots,P_k$ be the vertical paths contained in the root bag (note that $k\le 32g$) and
let $\TT_1,\ldots,\TT_{\ell}$ be the subtrees rooted at the children of the root bag (note that $\ell\le 108g$).
Further, let $V_i$, $i\in [\ell]$, be the vertices contained in the vertical paths in the bags of the subtree $\TT_i$ that
are not contained in the root bag,
i.e., the vertices that are not contained in $V(P_1)\cup\cdots\cup V(P_k)$.
Note that for every $i=1,\ldots,\ell$, the subgraph induced by a set $V_i$ has a component that
is joined by an edge to each of the paths $P_1,\ldots,P_k$ that are contained in the subtree $\TT_i$.

Let $H_0$ be the graph obtained from $G_0$ by contracting each of the following $k+\ell$ sets to a single vertex:
$V(P_1),\ldots,V(P_k)$ and $V_1,\ldots,V_\ell$.
Let $a_1,\ldots,a_k$ and $b_1,\ldots,b_{\ell}$ be the resulting vertices.
We observe that $H_0$ can be obtained from $G_0$ by contracting edges and deleting vertices.
The vertices $a_1,\ldots,a_k$ are obtained by contracting paths $P_1,\ldots,P_k$ and
we may think of each vertex $b_i$, $i=1,\ldots,\ell$, to be obtained as follows:
first contract the component of the subgraph induced by the set $V_i$ that
is joined by an edge to each of the paths $P_1,\ldots,P_k$ contained in the subtree $\TT_i$ to a single vertex, and
then delete all the vertices of $V_i$ not contained in this component.
Since $H_0$ can be obtained from $G_0$ by contracting edges and deleting vertices,
the graph $H_0$ can be embedded in the same surface as $G_0$.
Hence,
the number of the edges of $H_0$ is at most $3(k+\ell)-6+3g\le 3(k+\ell+g)$ (the latter bound applies even if $k+\ell=2$).
Since each subtree $\TT_i$ contains at most six of the paths $P_1,\ldots,P_k$,
each of the vertices $b_1,\ldots,b_{\ell}$ has degree at most six and
all its (at most six) neighbors are among the vertices $a_1,\ldots,a_k$.

Let $s=3\sqrt{47g}$; note that $s\ge 6$.
We next split the vertices $a_1,\ldots,a_k$ into sets $A_1,\ldots,A_{k'}$ and
the vertices $b_1,\ldots,b_{\ell}$ into sets $B_1,\ldots,B_{\ell'}$ as follows;
a similar argument has also been used in~\cite{AhnHKO22}.
Keep adding the vertices $b_1,\ldots,b_k$ to the set $B_1$ as long as the sum of their degrees does not exceed $s$,
then keep adding the remaining vertices among $b_1,\ldots,b_k$ to the set $B_2$ as long as the sum of their degrees does not exceed $s$, etc.
Observe that the sum of the degrees of the vertices in each of the sets $B_1,\ldots,B_{\ell'}$ is at most $s+6\le 2s$ and
the sum of the degrees of the vertices in each of the sets $B_1,\ldots,B_{\ell'-1}$ is at least $s$.
Each of the vertices $a_1,\ldots,a_k$ with degree larger than $s$ forms a set of size one, and
the remaining vertices are split in the same way as the vertices $b_1,\ldots,b_k$.
Each of the sets $A_1,\ldots,A_{k'}$ has either size one or the sum of the degrees of its vertices is at most $2s$, and
the sum of the degrees of the vertices in each of the sets $A_1,\ldots,A_{k'-1}$ is at least $s$.
Let $H'_0$ be the graph obtained from $H_0$ by contracting the vertices
in each of the sets $A_1,\ldots,A_{k'}$ and each of the sets $B_1,\ldots,B_{\ell'}$ to a single vertex;
note that the graph $H'_0$ does not need to be embeddable in the same surface as $H_0$.
Since the sum of the degrees of the vertices $a_1,\ldots,a_k$ and $b_1,\ldots,b_{\ell}$
is at most $6(k+\ell+g)\le 846g$ (as $H_0$ has at most $3(k+\ell+g)$ edges),
we obtain that $k'+\ell'\le\frac{846g}{s}+2=2s+2$,
i.e., $H'_0$ has at most $2s+2$ vertices.

We now describe the order in which we contract the vertices of $G$, and
we analyze the described order later.
In what follows, when we say a \emph{layer},
we always refer to the layers given by the BFS spanning tree $T_0$ from the application of Theorem~\ref{thm:product}.
In particular, each vertex of $G$ is adjacent only to the vertices in its own layer and the two neighboring layers.
To make the presentation of the order of contractions clearer,
we split the contractions into three phases.

\textbf{Phase I:}
This is the most complex phase and consists of $\ell$ subphases.
In the $i$-th subphase, $i\in [\ell]$,
we contract all the vertices of the set $V_i$ that are contained in the same layer to a single vertex
in the way that we now describe.
Then, we possibly contract them to some of the vertices created in the preceding subphases,
i.e., those obtained by contracting vertices in $V_1\cup\cdots\cup V_{i-1}$.
In this phase, \emph{we never contract two vertices contained in different layers} and
no contraction involves any vertex from $V(P_1)\cup\cdots\cup V(P_k)$.

\textbf{Subphase.}
Fix $i\in [\ell]$.
Let $G_i$ be the subgraph of $G_0/\PP$ induced by the vertices contained in the bags of the subtree $\TT_i$ and
let $n'$ be the number of the paths $P_1,\ldots,P_k$ that are contained in the bags of the subtree $\TT_i$;
note that $n'\le 6$.
If the graph $G_i$ has less than $8$ vertices, we proceed directly to the conclusion of the subphase, which is described below.
If the graph $G_i$ has at least $8$ vertices,
$G_i$ is a subgraph of a spanning subgraph of a $7$-tree $G'_i$ such that
the $n'$ vertices corresponding to the paths from the set $\{P_1,\ldots,P_k\}$
are contained in the initial complete graph of $G'_i$.

Fix any order $Q_1,\ldots,Q_n$ of the vertical paths corresponding to the vertices of $G'_i$ such that
the neighbors of $Q_j$, $j\in [n]$, among $Q_1,\ldots,Q_{j-1}$ form a complete graph of order at most $7$ in $G'_i$ and
the $n'$ paths from the set $\{P_1,\ldots,P_k\}$ are the paths $Q_1,\ldots,Q_{n'}$.
Let $C_j$ be the complete subgraph of $G_i$ formed by $Q_j$ and its (at most $7$) neighbors among $Q_1,\ldots,Q_{j-1}$.
Note that the neighbors in $G$ of each vertex of a path $Q_j$, $j\in [n]$,
are contained in at most seven of the paths $Q_1,\ldots,Q_{j-1}$,
which are exactly the paths forming the complete graph $C_j$.
We define the \emph{$j$-shadow} of a vertex $v\in V_i$
to be the set of its neighbors contained in the paths $Q_1,\ldots,Q_{j-1}$.
Since every vertex of $Q_j$ has at most $21$ neighbors on the paths $Q_1,\ldots,Q_{j-1}$ (as
its neighbors must be in the same or adjacent layers),
the $j$-shadow of a vertex contained in the path $Q_j$ has at most $21$ vertices.

We now use the tree-like structure of the $7$-tree $G'_i$ to define the order of contractions of the vertices contained in $V_i$;
this part of our argument is analogous to that used in~\cite{JacP22wg}
to obtain an upper bound on twin-width of graphs with bounded tree-width.
We proceed iteratively for $j=n-1,\ldots,n'$ as follows.
At the end of the iteration for $j=n-1,\ldots,n'+1$, all vertices of $V_i$ that
\begin{itemize}
\item are contained on paths in the same component of $G'_i\setminus\{Q_1,\ldots,Q_{j-1}\}$,
\item have the same $j$-shadow, and
\item are in the same layer
\end{itemize}
will have been contracted to a single vertex.
At the end of the iteration for $j=n'$,
all vertices of $V_i$ with the same $(n'+1)$-shadow that are contained in the same layer
will have been contracted to a single vertex,
in particular,
all vertices of $V_i$ contained in the same layer will have been contracted to at most $2^{3(n'-1)}$ vertices.

Fix $j\in\{n-1,\ldots,n'\}$.
Let $m$ be the number of components of $G'_i\setminus\{Q_1,\ldots,Q_{j}\}$ that
are included in the component of $G'_i\setminus\{Q_1,\ldots,Q_{j-1}\}$ containing $Q_j$, and
let $W_1,\ldots,W_m$ be the sets of vertices obtained by contracting vertices on the paths of these $m$ components.
Note that each $W_i$ has at most $2^{21}$ vertices in each layer.
We first contract each vertex of $W_2$
to the vertex $W_1$ with the same $(j+1)$-shadow on the same layer if such vertex exists.
Next, we contract each vertices of $W_3$ to the vertex of $W_1\cup W_2$
with the same $(j+1)$-shadow on the same layer if such vertex exists, etc.
At the end of this process,
all vertices of $W_1\cup\cdots\cup W_m$ with the same $(j+1)$-shadow that are on the same layer
have been contracted to a single vertex (note that there are at most $2^{24}$ such vertices in each layer).
If $j>n'$,
we contract all resulting vertices with the same $j$-shadow that are on the same layer to a single vertex, and
subsequently,
we contract the vertex contained on the path $Q_j$ to the vertex with the same $j$-shadow on the same layer (if such vertex exists).
The description of the iteration for $j$ is now finished.

\textbf{Conclusion of subphase.}
The $i$-th subphase concludes by contracting all the vertices of $V_i$ in the same layer to a single vertex, and
if the vertex $b_i$ is not the vertex with the smallest index in the set $B_{i'}$ such that $b_i\in B_{i'}$,
then we contract the resulting vertices to those in the same layers that have been obtained in the subphases
associated with the vertices of $B_{i'}$ with indices smaller than $i$.

\begin{figure}
\begin{center}
\epsfbox{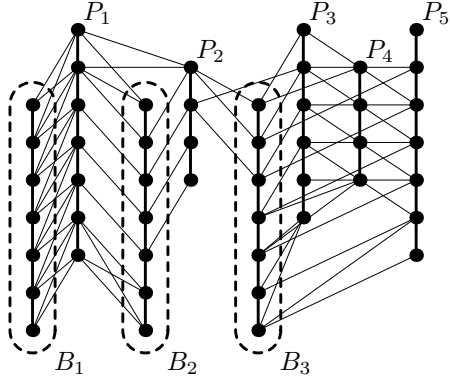}
\end{center}
\caption{An example of a graph obtained after Phase I in the proof of Theorem~\ref{thm:main} ($k=5$ and $\ell'=3$).
         The edges of vertical paths are drawn in bold.
	 Note that there are no edges between paths corresponding to the set $B_1$, $B_2$ and $B_3$.}
\label{fig:main1}
\end{figure}

\textbf{Phase II:}
The graph that we obtain after Phase I has at most $k+\ell'$ vertices in each layer:
$k$ of them corresponding to the vertices $a_1,\ldots,a_k$ of the graph $H_0$ and
the remaining $\ell'$ to the sets $B_1,\ldots,B_{\ell'}$ (see Figure~\ref{fig:main1}).
For every $i=1,\ldots,k'$,
we will contract all the vertices on the paths of $A_i$ that are in the same layer to a single vertex as follows.
Let $P_{i_1},\ldots,P_{i_n}$ be the paths corresponding to the vertices of $A_i$.
We first contract the vertices of $P_{i_1}$ and $P_{i_2}$ that are in the same layer,
proceeding from top to bottom (starting with the layer that contains both such vertices).
We then contract the vertices of $P_{i_3}$ to the vertices created previously,
again in each layer proceeding from top to bottom, then the vertices of $P_{i_4}$, etc.
At the end of this phase, we obtain a graph that is a subgraph of the strong product of a path and the graph $H'_0$.
Since the graph $H'_0$ has $k'+\ell'\le 2s+2$ vertices, each layer now contains at most $2s+2$ vertices.

\textbf{Phase III:}
We now contract all the vertices contained in the top layer to a single vertex,
then all the vertices of the next layer to a single vertex, etc.
Finally, we contract the vertices one after another to eventually obtain a single vertex,
starting with the two vertices of the top two layers,
then contracting the vertex in the third layer, etc.

\textbf{Analysis of red degrees.}
We now establish an upper bound on the maximum possible red degree of the vertices of the graphs
obtained throughout the described sequence of contractions.
We start with Phase I.
During the $i$-th subphase and the iteration for $j$,
the only new red edges ever created are among the vertices of $W_1,\ldots,W_m$ and the path $Q_j$.
Since the vertices of $W_1\cup\cdots\cup W_m$ have at most $2^{24}$ different $(j+1)$-shadows (the neighbors
in their shadows are only on the paths contained in $C_j$),
each vertex has neighbors in its and the two neighboring layers, and
we first contract all vertices of $W_1\cup W_2$ with the same $(j+1)$-shadow, then all vertices of $W_1\cup W_2\cup W_3$, etc.,
the red degree of any vertex does not exceed $2\cdot 3\cdot 2^{24}=3\cdot 2^{25}$.
We eventually arrive at having at most $2^{24}$ vertices in each layer and
so their red degrees do not exceed $3\cdot 2^{24}$.
Then, the vertices with the same $j$-shadow that are on the same layer are contracted,
which can result in the vertices of $Q_j$ (temporarily) having the red degree up to $3\cdot 2^{24}$.
At the end of iteration for $j>n'$,
there are at most $2^{21}$ vertices in each layer that have been obtained from $W_1\cup\cdots\cup W_m$ and
so the red degree of each of them is at most $3\cdot 2^{21}$.
Also note that there is no red edge between the vertices on the paths $Q_1,\ldots,Q_{j-1}$ and
the remaining vertices of $V_i$.

At the beginning of the conclusion of the subphase,
each layer has at most $2^{18}$ vertices obtained from contracting the vertices of $V_i$ (note that
this bound also holds when $G_i$ has less than eight vertices,
i.e., when we proceeded directly to the conclusion of the subphase).
The conclusion of the subphase starts with contracting these vertices to a single vertex per layer:
this can increase the red degree of at most six paths $P_1,\ldots,P_k$ and
the red degree of each vertex on the paths can increase by at most three.
When the subphase finishes,
each of the vertices contained in the paths $P_1,\ldots,P_k$ has at most $3\ell'$ red neighbors (although during the subphase it can have upto three additional red neighbors), and
each of the vertices obtained by contracting the vertices of $V_1,\ldots,V_i$ has red degree at most $6s$ (since
the sum of the degrees of the vertices in each set $B_1,\ldots,B_{\ell'}$ is at most $2s$).
In particular, the red degree of each vertex on the paths $P_1,\ldots,P_k$ never exceeds $3(\ell'+1)$.
We conclude that the red degree of none of the vertices exceeds
the largest of the following three bounds: $3\cdot 2^{25}$, $3(\ell'+1)$ and $6s$.
Moreover,
the red degree of no vertex exceeds $\max\{3\ell',6s\}\le 6s+3$ (recall that $\ell'\le 2s+1$)
at the end of each subphase (and so also at the end of Phase I).

During Phase II,
each vertex has at most $\max\{k'+\ell',2s\}$ red neighbors in its layer and in each of the neighboring layers.
Indeed,
the vertices obtained from those on the paths $P_1,\ldots,P_k$
have at most $k'+\ell'$ red neighbors in each layer (at most $k'$ neighbors
among vertices obtained from contracting vertices on the paths $P_1,\ldots,P_k$, and
there are at most $\ell'$ vertices in each layer obtained by contracting vertices not on the paths $P_1,\ldots,P_k$) and
the vertices obtained from those not on the paths $P_1,\ldots,P_k$ have at most $2s$ red neighbors in each layer as
this is simply the upper bound on the number of their neighbors on the paths $P_1,\ldots,P_k$.
Hence, the red degree of any vertex never exceeds
\[3\max\{k'+\ell',2s\}\le 3\max\{2(s+1),2s\}=6(s+1)\]
during the entire Phase II.
Finally, since the number of vertices contained in each layer at the end of Phase II is at most $k'+\ell'$,
during the entire Phase III, the red degree of no vertex exceeds $3(k'+\ell')-1$.

Hence, we have established that the red degree of no vertex exceeds $\max\left\{6(s+1),3\cdot 2^{25}\right\}$,
which implies the bound claimed in the statement of the theorem.
\end{proof}

\section{Lower bound}
\label{sec:lower}

The following can be readily obtained from the result of Ahn et al.~\cite{AhnCHKO22} on twin-width of random graphs,
however, we include a short proof for completeness.

\begin{proposition}
\label{prop:lower}
There exists a graph of Euler genus $g$ that has twin-width at least $\sqrt{3g/2}-O(g^{3/8})$.
\end{proposition}

\begin{proof}
Fix $g>2$. It is well-known~\cite{RinY68} that the complete graph of order $n$ where $n$ is equal to the Heawood number,
i.e.,
\[n=\left\lfloor\frac{7+\sqrt{1+24g}}{2}\right\rfloor=\sqrt{6g}+O(1)\]
can be embedded in any surface of Euler genus $g$ (the reason why we excluded $g=2$ is that
this is not true for the Klein bottle).
Let $G$ be the $n$-vertex Erd\H os-R\'enyi random graph for $p=1/2$,
i.e., the graph $G$ such that any pair of vertices of $G$
is joined by an edge with probability $1/2$ independently of other pairs.
The probability that the degree of a particular vertex differs from $(n-1)/2$ by more than $n^{3/4}$
is at most $2e^{-\frac{2n^{3/2}}{n-1}}$ by the Chernoff Bound, and
the probability that the number of common neighbors of any particular pair of vertices differs from $(n-2)/4$ by more than $n^{3/4}$
is at most $2e^{-\frac{2n^{3/2}}{n-2}}$.
Hence, the probability that the degree of each vertex is between $\frac{n-1}{2}-n^{3/4}$ and $\frac{n-1}{2}+n^{3/4}$ and
that the number of common neighbors of each pair of vertices is between $\frac{n-2}{4}-n^{3/4}$ and $\frac{n-2}{4}+n^{3/4}$
is at least
\[1-2ne^{-\frac{2n^{3/2}}{n-1}}-2\binom{n}{2}e^{-\frac{2n^{3/2}}{n-2}}.\]
Assume that $n$ is sufficiently large that this probability is positive and
fix such an $n$-vertex graph $G$.
Clearly, $G$ can be embedded in any surface of Euler genus $g$ as it is a subgraph of $K_n$.
On the other hand, the contraction of any pair of vertices results in a vertex with red degree at least
\[2\frac{n-1}{2}-2n^{3/4}-2\frac{n-2}{4}-2n^{3/4}-2=\frac{n}{2}-2-4n^{3/4}=\sqrt{3g/2}-O(g^{3/8}).\]
The statement of the proposition now follows.
\end{proof}

\section*{Acknowledgement}

The substantial part of the work presented in this article was done during the Brno--Koper
research workshop on graph theory topics in computer science held in Kranjska Gora in April 2023,
which all three authors have participated in.

\bibliographystyle{bibstyle}
\bibliography{twing}

\end{document}